\documentclass[11pt]{article}

\usepackage{array, enumerate, color, url}
\usepackage{amsmath, epsfig}
\usepackage{latexsym}
\usepackage{mathrsfs}
\usepackage{amssymb}
\usepackage{graphicx}
 \usepackage{epstopdf}

\usepackage{showkeys}

 

\newtheorem{theorem}{Theorem}[section]
\newtheorem{lemma}[theorem]{Lemma}

\newcommand{\n}{\noindent}

\newtheorem{thm}{Theorem}[section]

\newtheorem{cor}[thm]{Corollary}

\def\P{{\mathcal P}}

\newenvironment{proof}{{\bf Proof}.}{\rule{3mm}{3mm}}
\makeatletter

\newcommand{\Rmnum}[1]{\expandafter\@slowromancap\romannumeral #1@}
\makeatother

\title{ Decompositions of  graphs of nonnegative characteristic with  some forbidden subgraphs}

\author{Lin Niu and Xiangwen Li\thanks{Supported in part by  NSFC (12031018)}\\
 School of Mathematics $\&$ Statistics\\
Central China Normal University, Wuhan 430079, China
}

\date{}

\begin{document}
\maketitle

\begin{abstract}
 A {\em  $(d,h)$-decomposition} of a graph $G$ is an order pair $(D,H)$ such that $H$ is a subgraph of $G$ where $H$ has the maximum degree at most $h$ and $D$ is an acyclic orientation of $G-E(H)$ of maximum out-degree at most $d$. A graph $G$ is {\em  $(d, h)$-decomposable} if $G$ has a  $(d,h)$-decomposition. Let $G$ be a graph embeddable in a surface of nonnegative characteristic.
In this paper, we prove the following results. (1) If $G$ has no  chord $5$-cycles or no chord $6$-cycles or no chord $7$-cycles and no adjacent $4$-cycles, then $G$  is $(3,1)$-decomposable, which generalizes the results of Chen, Zhu and Wang [Comput. Math. Appl, 56 (2008) 2073--2078] and the results of Zhang [Comment. Math. Univ. Carolin,  54(3) (2013) 339--344]. (2) If $G$ has no $i$-cycles nor $j$-cycles for any subset $\{i,j\}\subseteq \{3,4,6\}$  is $(2,1)$-decomposable, which generalizes the  results of Dong and Xu [Discrete  Math. Alg. and Appl., 1(2) (2009), 291--297].
\end{abstract}

\section{Introduction}

Graphs considered here are finite and simple. A graph  is {\em $d$-generate} if every subgraph has a vertex of degree at most $d$.  For two integers $d, h\in \mathbb{N}$, a {\em  $(d, h)$-decomposition} of $G$ is a pair $(H_1, H_2)$ such that $H_2$ is a subgraph of $G$ of maximum degree at most $h$ and $H_1$ is $d$-degenerate.
A graph $G$ is {\em  $(d, h)$-decomposable} if $G$ has a $(d, h)$-decomposition.
Decomposing a graph into subgraphs with simple structure is a fundamental problem in graph theory. The classical Theorem of Tutte~\cite{Tu} and, independent by,  Nash-Williams~\cite{NW}  provides a necessary and sufficient condition for which a graph can be decomposed into forests. A proper coloring of $G$ is  a decomposition of $G$ into independent sets. The problems of decomposing a graph $G$ into start forests, linear forests and some others are studied widely in the literature.

A {\em proper $k$-coloring} is a mapping $\varphi: V(G)\rightarrow \{1,2,\ldots,k\}$ such that $\varphi(u)\neq \varphi(v)$ where $uv\in E(G)$. The {\em chromatic number}, denoted by $\chi (G)$, of $G$ is the minimum $k$ such that $G$ is $k$-colorable.  A {\em $d$-defective $k$-coloring} of $G$ is a mapping $\varphi: V(G)\rightarrow \{1,2,\ldots,k\}$ such that for each vertex $v\in V(G)$, $v$ has at most $d$ neighbors of the same color as itself.
A {\em $k$-list assignment} of $G$ is a function $L$ that assigns a list $L(v)$ of colors to each vertex $v \in V(G)$ where $|L(v)|=k$. A {\em $d$-defective $L$-coloring} is a mapping $\varphi$ that assigns a color $\varphi(v) \in L(v)$ to
each vertex $v \in V(G)$ such that $v$ has at most $d$ neighbors of the same color as itself. A graph $G$ is {\em $d$-defective $k$-choosable} if there exists an $(L, d)$-coloring for every list assignment $L$ with $|L(v)| = k$ for all $v \in V(G)$. A graph is
 0-defective $k$-choosable if and only if it is  $k$-choosabe.  The {\em choosable number}, denoted by $ch(G)$, of $G$ is the minimum $k$ such that $G$ is $k$-choosable.

 The {\em Alon-Tarsi number} of $G$, denoted by $AT(G)$, was defined by Jensen and Toft \cite{Jen}.
 It follows from the Alon-Tarsi Theorem \cite{Alo1} that $ch(G)\leq AT(G)$ for any graph $G$. It is proved that the difference $AT(G)-ch(G)$ can be arbitrarily large.
 DP-coloring was introduced by Dvo\v{r}\'{a}k and Postle \cite{DP} as a generation of list coloring.  Clearly, $ch(G)\leq \chi_{DP}(G)$, where  $\chi_{DP}(G)$ is the DP-chromatic number of a graph $G$. A painting coloring was introduced by Schauz~\cite{US14} and it is proved that $ch(G) \leq \chi_P(G)$ for any graph $G$, where $\chi_P(G)$ is the paint number of $G$.

It is well-known that a graph $H_1$  has an acyclic orientation $D$ with $\Delta_D^+\le d$ if and only if $H_1$ is $d$-degenerate, where $\Delta_D^+$ is the maximum degree $D$.
 If $G$ is  $d$-degenerate, then each of choosable number $ch(G)$, Alon-Tarsi number $AT(G)$, paint number $\chi_P(G)$ and DP-chromatic number $\chi_{DP}$ is at most $d+1$. This implies that if $G$ is $(d,h)$-decomposable, then there is a subgraph $H$ of $G$ where $\Delta_H\le h$ such that $G-E(H)$ is $h$-defective-$(d+1)$-choosable, $(d+1)$-DP-colorable and $AT(G-E(H))\leq d+1$.

Defective coloring of graphs was considered by Cowen, Cowen and Woodall \cite{Cow} who proved that every planar graph is 2-defective 3-colorable, which was improved by Eaton and Hull \cite{EH}, independently, \v{S}krekovski~\cite{S}, who proved that every planar graph is 2-defective 3-choosable.  Cushing and Kierstead \cite{Cush} proved that every planar graph is 1-defective 4-choosable. Grytczuk and Zhu \cite{Gry} strengthen the result and  proved that every planar graph $G$ has a matching $M$ such that $AT(G-M)\leq4$.
Lih, Song, Wang and Zhang \cite{L} proved that  every planar  graph $G$ without $4$-cycles and $l$-cycles for some $l\in\{5,6,7\}$ is 1-defective 3-choosable. Dong and Xu \cite{Xu} showed that such result is also true for some $l\in\{8, 9\}$.
Lu and Zhu  \cite{LZ20} proved that   every planar graph without $4$- and $l$-cycles $G$, where $l=5, 6, 7$,  has a matching $M$ such that $G-M$ is  $AT(G-M)\le3$.  Gon\c{c}alves \cite{Go09} proved that every planar graph is $(3, 4)$-decomposable. Zhu \cite{Zhu00} proved that every planar graph is $(2, 8)$-decomposable.
Recently,
Li, Lu,  Wang and  Zhu \cite{Zhu} improve  this result and prove that for $l\in\{5,6,7,8,9\}$, every planar graph without $4$- and $l$-cycles is $(2,1)$-decomposable. Cho, Choi, Kim, Park, Shan and Zhu \cite{Zhu21} prove that every planar graph is $(4, 1)$-, $(3,2)$-, $(2, 6)$-decomposable and that there are  planar graphs which are not $(2,3)$-decomposable and there are also planar graphs which are not $(1, h)$-decomposable.

We are interested in  decompositions of  graphs of nonnegative characteristic in this paper.  The characteristic of a surface $\Sigma$ is defined to be $|V(G)|-|E(G)|+|F(G)|$ for any graph $G$ which is 2-cell embedded in $\Sigma$. All the surfaces of nonnegative characteristic are the Euclidean plane, the projective plane, the torus and the Klein bottle.  A graph of nonnegative characteristic means that it  can be embedded on a surface of nonnegative characteristic.  Throughout this paper, a graph of nonnegative characteristic is called a {\em NC-graph}. In this paper, we prove the following results.

 \begin{thm}\label{th0}
A NC-graph $G$ is $(3,1)$-decomposable if one of the following hold:

 (1) $G$ has no chord 5-cycles.

 (2) $G$ has no chord $6$-cycles.

 (3) $G$ has no  chord 7- nor adjacent 4-cycles.
\end{thm}

For simplicity, we define a family $\cal G$ of NC-graphs   such that $G\in {\cal G}$ if and only if $G$ has neither chord $5$-cycles nor chord $6$-cycles nor chord $7$- and adjacent $4$-cycles. From Theorem \ref{th0}, next corollary follows immediately.

\begin{cor}
 Every graph  $G \in {\cal G}$  has a matching $M$ such that each of choice number, paint number, DP-number and Alon-Tarsi number of $G-M$ is at most 4.
\end{cor}

A graph $G$ is {\em toroidal} if $G$ can be drawn on the torus so that the edges meet only at the vertices of the graph.
\begin{cor}
(1) (Chen, Zhu and Wang, \cite{Chen}) Every  graph of nonnegative characteristic without either chord 5-cycles or chord 6-cycles is 1-defective 4-choosable.

(2) (Zhang, \cite{Zhang})  Every toroidal graph $G$ without chord 7-cycles and adjacent 4-cycles is 1-defective 4-choosable.
\end{cor}

 \begin{thm}\label{th1}
A NC-graph $G$ is $(2,1)$-decomposable if one of the following hold:

  (1) $G$ has neither $3$- nor $4$-cycles.

  (2) $G$ has neither $3$- nor $6$-cycles.

  (3) $G$ has neither $4$- nor $6$-cycles.
\end{thm}

Similarly, we define a family $\cal H$ of NC-graphs  such that $G\in {\cal H}$ if and only if $G$ has no $i$-cycles nor $j$-cycles for any  $\{i,j\}\subseteq \{3,4,6\}$.  We obtain the following corollary  from Theorem \ref{th1}.
\begin{cor}
Every graph  $G \in {\cal H}$ has a matching $M$ such that each of choice number, paint number, DP-number and Alon-Tarsi number of $G-M$ is at most 3.
\end{cor}

\begin{cor}\label{cor1}
(Dong and Xu \cite{Xu})  Every  toroidal graph $G$ which contains neither $i$-cycles nor $j$-cycles for any subset $\{i,j\}\subseteq \{3,4,6\}$ is 1-defective 3-colorable.
\end{cor}

In the end of this section, we introduce some terminology and notation. Let $G$  be a  graph and denote by $V(G), E(G), F(G)$ (or $V, E, F$ for short) the sets of vertices, edges and faces of $G$, respectively.  Let $G$ be a  (di)graph.  For a vertex $v$,  denote by $d(v)$ ($d^+(v)$ or $d^-(v)$ in digraph)  the {\em degree} ({\em out-degree} or {\em in-degree} in digraph) of $v$. Denote by  $N_G(v)$ (or $N(v)$ for short) the set of {\em neighbors of a vertex} $v$ in $G$.
 A {\em  $k$-vertex} ({\em  $k^{+}$-vertex} or {\em  $k^{-}$-vertex}) is a vertex of degree $k$ ( at least $k$ or at most $k$). Similarly, a {\em  $k$-face} ({\em  $k^{+}$-face} or {\em  $k^{-}$-face}) is a face of degree $k$ (at least $k$ or at most $k$). For $f\in F(G)$, denote by $d(f)$  the {\em degree} of face $f$ in $G$ which is the number of edges incident with  $f$
   and $b(f)$  the {\em boundary walk of f} and write $f=[u_1u_2\ldots u_l]$ when $u_1, u_2, \ldots, u_l$ are the boundary vertices of $f$ in clockwise order.  A $l$-face $[u_1u_2\ldots u_l]$ is called an $(a_1, a_2,\ldots, a_l)$-face if $d(u_i) = a_i$ for $i =1, 2,\ldots, l$.
   Two faces are {\em  adjacent} if they share at least one common edge.
For $v\in V(G)$ and $i\ge3$, denote by $n_i(v)$ ($n_{i^+}(v)$ or ($n_{i^-}(v)$) the number of all $i$- ($i^+$- or ($i^-$-) faces incident to $v$.
   A cycle  is a $k$-cycle if it contains $k$ vertices. For a cycle $C$,  an edge $xy\in E(G)\setminus E(C)$ is called a {\em chord} of $C$ if $ x, y \in V(C)$. Let $C$ be a $k$-cycle. Then $C$ is called {\em chord $k$-cycle}.

\section{Reducible configurations}

Suppose otherwise that Theorems~\ref{th0} and \ref{th1}  are both false. Assume that
\begin{equation}\label{e1}
\mbox{ $G\in {\cal G}$ is a  counterexample to Theorem \ref{th0}   with  $|V(G)|$ minimized.}
\end{equation}
In this case, $G$  has no chord 5-cycles nor chord 6-cycles nor chord 7-cycles and adjacent 4-cycles.  Clearly, $G$ has no $(3,1)$-decomposition but any subgraph of $G$ does. Similarly,  assume that
 \begin{equation} \label{e2}
 \mbox{ $H\in {\cal H}$ is a  counterexample to Theorem \ref{th1}  with  $|V(H)|$ minimized. }
\end{equation}
In this case, $H$ has neither $i$-cycle nor $j$-cycle, where $\{i, j\}\subset \{3,4, 6\}$.  Clearly, $H$ has no $(2,1)$-decomposition but any subgraph of $H$ does. In this section,  we establish several lemmas. The following lemma is straightforward.

  \begin{lemma}\label{lem0} Assume that $G$ is a NC-graph and
   $d(v)\geq3$ for all $v\in V(G)$. If $G$ has no $6$-cycles, then two $4$-faces are not adjacent.
  \end{lemma}

Recall that a graph $H$ is $d$-degenerate if and only if $H$  has an acyclic orientation $D$ with $\Delta_D^+\le d$. Thus, to prove that a graph $G$ is $(d, h)$-decomposable, it is sufficient to show that $G$ can be decomposed into $H_1$ and $H_2$ such that $H_1$ has an acyclic orientation with $\Delta_D^+\leq d$ and $H_2$ has the maximum degree at most $k$. From Lemma~\ref{lem1} to \ref{lem4}, we assume that $G$ satisfies Assumption (\ref{e1}).

\begin{lemma}\label{lem1}
(1) $d(v)\geq4$ for all $v\in V(G)$;

(2) $G$ does not contain two adjacent $4$-vertices.
\end{lemma}
\begin{proof}
(1) Suppose otherwise that $v$ is a $3$-vertex and $N(v)=\{v_1,v_2,v_3\}$. By the minimality of $G$, there is a $(3,1)$-decomposition $(D^*, M^*)$ of $G-\{v\}$.  Let $M=M^*$ and $D=D^*\cup \{\overrightarrow{vv_1},\overrightarrow{vv_2}, \overrightarrow{vv_3}\}$. Then $(D,M)$ is a $(3,1)$-decomposition of $G$, a contradiction.

(2) Suppose otherwise that $u$ is a 4-vertex adjacent to a $4$-vertex $v$. Let $N(u)=\{u_1,u_2,u_3,v\}$ and $N(v)=\{v_1, v_2, v_3,u\}$. By the minimality  of $G$, there is a $(3,1)$-decomposition $(D^*, M^*)$ of $G-\{u,v\}$. Let $M=M^*\cup \{uv\}$ and $D=D^*\cup\{\overrightarrow{vv_1},\overrightarrow{vv_2}, \overrightarrow{vv_3}, \overrightarrow{uu_1},\overrightarrow{uu_2}, \overrightarrow{uu_3}\}$. Then $(D,M)$ is a $(3,1)$-decomposition of $G$, a contradiction.
\end{proof}

\begin{lemma}\label{lem2}
(1) A $5$-vertex $v$ is incident with at most one $(4,5,5)$-face.

(2) A $5$-vertex $v$ is  not incident with three consecutively adjacent 3-faces, one of which is $(4,5,5)$-face and other two of which are $(4,5,6)$-faces.
\end{lemma}
\begin{proof}
 Let $v_1, v_2, \ldots, v_5$ be the neighbors of $v$ in clockwise order, and $f_1,  f_2, \ldots , f_5$ be the incident faces of $v$ with $vv_i, vv_{i+1}\in b(f_i)$ for $i=1, 2, \ldots, 5$ where indices are taken modulo 5.

 (1) Suppose otherwise that $v$ is incident with two $(4,5,5)$-faces.
 There are two cases.

\n{\bf Case 1.} $f_1$ and $f_2$ are $(4,5,5)$-faces.

 We first assume that $d(v_1)=d(v_3)=5$ and $d(v_2)=4$. Let $N(v_1)=\{v_{11}, v_{12}, v_{13}, v, v_2\}, N(v_2)=\{v_{21},v,v_1,v_3\}$ and $N(v_3)=\{v_{31},v_{32},v_{33},v, v_2\}$.

By the minimally of $G$, there is a $(3,1)$-decomposition $(D^*, M^*)$ of $G-\{v,v_1, v_2,v_3\}$.
Let
$M=M^*\cup \{vv_1, v_2v_3\}$ and $D=D^*\cup
\{\overrightarrow{v_1v_{11}},\overrightarrow{v_1v_{12}}, \overrightarrow{v_1v_{13}}, \overrightarrow{v_2v_{21}},\overrightarrow{v_3v_{31}}, \overrightarrow{v_3v_{32}}, \overrightarrow{v_3v_{33}}, \overrightarrow{vv_3}, \overrightarrow{vv_4}, \overrightarrow{vv_5}, \overrightarrow{v_2v},\overrightarrow{v_2v_1}\}$.
 Then $(D,M)$ is a $(3,1)$-decomposition of $G$, a contradiction.

  We further assume that $d(v_1)=d(v_3)=4$ and $d(v_2)=5$. Let $N(v_1)=\{v_{11}, v_{12}, v, v_2\}, N(v_2)=\{v_{21}, v_{22},v,v_1,v_3\}$ and $N(v_3)=\{v_{31},v_{32},v, v_2\}$.

By the minimality of $G$, there is a $(3,1)$-decomposition $(D^*, M^*)$ of $G-\{v,v_1, v_2,v_3\}$.
Let $M=M^*\cup \{v_1v_2, vv_3\}$ and $D=D^*\cup
\{\overrightarrow{v_1v_{11}},\overrightarrow{v_1v_{12}},  \overrightarrow{v_2v_{21}}, \overrightarrow{v_2v_{22}},\overrightarrow{v_3v_{31}}, \overrightarrow{v_3v_{32}},  \overrightarrow{vv_2}, \overrightarrow{vv_4}, \overrightarrow{vv_5},  \overrightarrow{v_1v},\overrightarrow{v_3v_2}\}$.
 Then $(D,M)$ is a $(3,1)$-decomposition of $G$, a contradiction.

\n{\bf Case 2.} $f_1$ and $f_3$ are $(4,5,5)$-faces.

We assume, without loss of generality, that $d(v_1)=d(v_3)=4, d(v_2)=d(v_4)=5$ and $N(v_1)=\{v_{11},v_{12},v,v_2\}, N(v_2)=\{v_{21}, v_{22}, v_{23}, v, v_1\}, N(v_3)=\{v_{31},v_{32},v,v_4\}, N(v_4)=\{v_{41},v_{42},v_{43},v,v_3\}$.

By the minimality of $G$,  there is a $(3,1)$-decomposition $(D^*, M^*)$ of $G-\{v,v_1, v_2,v_3, v_4\}$.
Let $M=M^*\cup \{v_1v_2, v_3v_4\}$ and $D=D^*\cup
\{\overrightarrow{v_1v_{11}},\overrightarrow{v_1v_{12}}, \overrightarrow{v_2v_{21}}, \overrightarrow{v_2v_{22}}, \overrightarrow{v_2v_{23}},\overrightarrow{v_3v_{31}}, \overrightarrow{v_3v_{32}},  \overrightarrow{v_4v_{41}}, \overrightarrow{v_4v_{42}}, \\ \overrightarrow{v_4v_{43}},
\overrightarrow{vv_2}, \overrightarrow{vv_4}, \overrightarrow{vv_5}, \overrightarrow{v_1v},  \overrightarrow{v_3v}\}$.
 Then $(D,M)$ is a $(3,1)$-decomposition of $G$, a contradiction.

 (2) By (1) and by symmetry, suppose otherwise that $f_1$ is a $(4,5,5)$-face and $f_2, f_3$ are two $(4,5,6)$-faces. In this case, $d(v_1)=5, d(v_2)=d(v_4)=4$ and $d(v_3)=6$. Let $N(v_1)=\{v_{11}, v_{12}, v_{13}, v, v_2\}, N(v_2)=\{v_{21}, v, v_1, v_3\}, N(v_3)=\{v_{31}, v_{32}, v_{33}, v, v_2, v_4\}$ and $N(v_4)=\{v_{41}, v_{42}, \\v, v_3\}$.
By the minimality of $G$,  there is a $(3,1)$-decomposition $(D^*, M^*)$ of $G-\{v,v_1, v_2,v_3, v_4\}$.
Let $M=M^*\cup \{v_1v_2, v_3v_4\}$ and $D=D^*\cup
\{\overrightarrow{v_1v_{11}},\overrightarrow{v_1v_{12}}, \overrightarrow{v_1v_{13}}, \overrightarrow{v_2v_{21}}, \overrightarrow{v_3v_{31}}, \overrightarrow{v_3v_{32}}, \overrightarrow{v_3v_{33}}, \overrightarrow{v_4v_{41}}, \overrightarrow{v_4v_{42}}, \\ \overrightarrow{vv_1}, \overrightarrow{vv_3},  \overrightarrow{vv_5}, \overrightarrow{v_2v}, \overrightarrow{v_2v_3}, \overrightarrow{v_4v}\}$.
 Then $(D,M)$ is a $(3,1)$-decomposition of $G$, a contradiction.
\end{proof}

\begin{lemma}\label{lem3}
  If $G$ is a NC-graph without either chord $5$-cycles or chord $7$- and adjacent $4$-cycles, then every $4^+$-vertex $v$ is incident with at most two consecutively adjacent 3-faces.  Moreover, $v$ is incident with at most $\lfloor\frac{2d(v)}{3}\rfloor$ 3-faces.
 \end{lemma}
 \begin{proof}
 Suppose otherwise that  $v$ is a $4^+$-vertex incident with three consecutively adjacent 3-faces $[v_1vv_2], [v_2vv_3]$ and $[v_3vv_4]$.
In this case,  $G$ has a $5$-cycle $[v_1v_2v_3v_4v]$ with a chord $vv_3$, a contradiction.
Observe that two adjacent 4-faces $f_1=[v_1v_2v_3v], f_2=[v_2v_3v_4v]$  have one common edge $v_2v_3$.  Thus, $G$ has adjacent $4$-cycles, a contradiction.
 Therefore,   $v$ is  incident with at most $\lfloor\frac{2d(v)}{3}\rfloor$ 3-faces.
\end{proof}

\begin{lemma}\label{lem4}
Let $G$ be a NC-graph without chord 6-cycles. Then  every $5^+$-vertex $v$ is incident to at most three consecutively adjacent 3-faces.  Thus, $v$ is incident to at most $(d(v)-2)$ 3-faces.
\end{lemma}
\begin{proof}
  Suppose otherwise that  $v$ is a $5^+$-vertex incident to four consecutively adjacent 3-faces $[v_1vv_2], [v_2vv_3], [v_3vv_4], [v_4vv_5]$. Then $[v_1v_2v_3v_4v_5v]$ is a $6$-cycle with a chord $vv_3$, a contradiction.
Thus, $v$ is incident to at most  $(d(v)-2)$ 3-faces.
\end{proof}

\medskip

From Lemma~\ref{lem5} to \ref{lem6}, we assume that $G$ satisfies Assumption (\ref{e2}).

\begin{lemma}\label{lem5}
(1) $d(v)\geq3$ for all $v\in V(G)$;

(2) $G$ does not contain two adjacent $3$-vertices.
\end{lemma}
\begin{proof}
(1) Suppose otherwise that $v$ is a $2$-vertex and $N(v)=\{v_1,v_2\}$. By the minimality of $G$, there is a $(2,1)$-decomposition $(D^*, M^*)$ of $G-\{v\}$.
 Let $M=M^*$ and $D=D^*\cup\{\overrightarrow{vv_1},\overrightarrow{vv_2}\}$. Then $(D,M)$ is a $(2,1)$-decomposition of $G$, a contradiction.

(2) Suppose otherwise that  $u$ is a 3-vertex adjacent to a $3$-vertex $v$. Let $N(u)=\{u_1,u_2,v\}$ and $N(v)=\{v_1, v_2,u\}$. By the minimally of $G$, there is a $(2,1)$-decomposition $(D^*, M^*)$ of $G-\{u,v\}$. Let $M=M^*\cup \{uv\}$ and $D=D^*\cup\{\overrightarrow{vv_1},\overrightarrow{vv_2}, \overrightarrow{uu_1},\overrightarrow{uu_2}\}$. Then $(D,M)$ is a $(2,1)$-decomposition of $G$, a contradiction.
\end{proof}

\begin{lemma}\label{lem7}
If A NC-graph $G$ has has no 3-cycle nor 6-cycle, then it has no any underlying subgraph of $G$ in Fig.1.
\end{lemma}
\begin{proof}
 Suppose otherwise that $G$ contains one of the figures in Fig.1. Let $X$ be  all the labeled vertices of each figure. By the minimality of $G$, $G^*=G-X$ has a $(2,1)$-decomposition $(D^*, M^*)$.

In Fig.1 (1),  $X=\{v_1, \ldots, v_{11}\}$. Let $M'=\{v_1v_5, v_2v_3, v_6v_7,v_8v_9, v_{10}v_{11}\}$ and $D'=\{
\overrightarrow{v_1v_2},\overrightarrow{v_1v_7},
\\ \overrightarrow{v_2v_9}, \overrightarrow{v_3v_4},\overrightarrow{v_3v_8},\overrightarrow{v_4v_5},\overrightarrow{v_5v_6},
\overrightarrow{v_8v_{10}}, \overrightarrow{v_{11}v_4}\}$.
In Fig.1 (2),  $X=\{v_1, \ldots, v_{11}\}$. Let $M'=\{v_1v_5, v_2v_3, v_6v_7,\\ v_8v_9, v_4v_{11}\}$ and $D'=\{
 \overrightarrow{v_1v_2},\overrightarrow{v_1v_7},
\overrightarrow{v_2v_9},\overrightarrow{v_3v_4},\overrightarrow{v_3v_8},\overrightarrow{v_4v_5},\overrightarrow{v_5v_6},
\overrightarrow{v_8 v_{10}},\overrightarrow{v_{10}v_{11}}\}$.
In Fig.1 (3),  $X=\{v_1, \ldots, v_{11}\}$. Let $M'=\{v_1v_5,v_6v_7,v_2v_{11},v_4v_8, v_9v_{10}\}$ and $D'=\{\overrightarrow{v_1v_2}, \overrightarrow{v_1v_7},
\overrightarrow{v_2v_3},\overrightarrow{v_3v_4},\overrightarrow{v_3v_9},\overrightarrow{v_4v_5},\\ \overrightarrow{v_5v_6},
\overrightarrow{v_9v_8}, \overrightarrow{v_{11}v_{10}}\}$.
In Fig.1 (4),   $X=\{v_1, \ldots, v_{11}\}$. Let $M'=\{v_1v_5,v_6v_7,v_2v_{11},v_4v_8, v_9v_{10}\}$ and $D'=\{
\overrightarrow{v_1v_2}, \overrightarrow{v_1v_7},
\overrightarrow{v_2v_3},\overrightarrow{v_3v_4},\overrightarrow{v_3v_9},\overrightarrow{v_4v_5},\overrightarrow{v_5v_6},
\overrightarrow{v_9v_8}, \overrightarrow{v_{10}v_{11}}\}$.
 In Fig.1 (5),  $X=\{v_1,v_2, \ldots,v_7, v_9\}$. Let $M'=\{ v_1v_2, v_4v_5, v_6v_7\}$ and $D'=\{
\overrightarrow{v_1v_5},\overrightarrow{v_1v_7},\overrightarrow{v_2v_6},
 \overrightarrow{v_2v_9}, \overrightarrow{v_3v_2},\overrightarrow{v_3v_4},\overrightarrow{v_7v_3},\overrightarrow{v_7v_9}\}$.
In Fig.1 (6),  $X=\{v_1, \ldots, v_9\}$. Let $M'=\{v_1v_5, v_3v_4,v_6v_7, v_8v_9\}$ and $D'=\{\overrightarrow{v_1v_2}, \overrightarrow{v_1v_7},
\overrightarrow{v_2v_6}, \overrightarrow{v_3v_2},\overrightarrow{v_3v_9},\overrightarrow{v_4v_8},\overrightarrow{v_5v_4},
\\ \overrightarrow{v_5v_9},\overrightarrow{v_6v_5}\}$.
In Fig.1 (7),  $X=\{v_1, \ldots, v_{11}\}$. Let
$M'=\{v_2v_3, v_4v_5, v_6v_7,v_8v_9, v_{10}v_{11}\}$ and $D'=\{ \overrightarrow{v_1v_5},\overrightarrow{v_1v_7},\overrightarrow{v_2v_1},
\overrightarrow{v_2v_6},\overrightarrow{v_3v_4},
\overrightarrow{v_3v_9}, \overrightarrow{v_4v_8},\overrightarrow{v_9v_{10}},\overrightarrow{v_{11}v_2}\}$.

Let $M=M^*\cup M'$ and $D$ be the orientation of $G-M$ obtained by adding arcs in $D'$ and
all the edges between $X$ and $V\setminus X$  oriented from $X$ to $V\setminus X$. Then $\Delta(M)\le 1$ and $\Delta_D^+\le2$.
Moreover, $D$ is an acyclic orientation of $G-M$. Thus $(D,M)$ is a $(2,1)$-decomposition of $G$, a contradiction.
\end{proof}

 \vskip -2cm
\unitlength=0.30mm
\begin{picture}(10,20)(0,0)
\put(50, -100){\makebox(0,0){$\bullet$}}
\put(55, -101){\scriptsize {\em $v_1$}}
\put(50, -100){\vector(2,-1){20}}
\put(50, -100){\line(-2,-1){30}}
\put(50, -100){\line(2,-1){30}}
\put(50, -100){\vector(0,1){20}}
\put(50, -100){\line(0,1){31}}
\put(50, -70){\makebox(0,0){$\bullet$}}
\put(55, -71){\scriptsize {\em $v_7$}}
\put(50, -70){\vector(-2,1){15}}
\put(50, -70){\vector(0,1){15}}
\put(20, -115){\makebox(0,0){$\bullet$}}
\put(13, -123){\scriptsize {\em $v_5$}}
\put(20, -115){\vector(-2,-1){15}}
\put(20, -115){\line(1,-2){12}}
\put(20, -115){\line(0,1){32}}
\put(20, -115){\vector(0,1){17}}
\put(20, -85){\makebox(0,0){$\bullet$}}
\put(10, -93){\scriptsize {\em $v_6$}}
\put(30, -80){\scriptsize {\em $M$}}
\put(30, -110){\scriptsize {\em $M$}}
\put(69, -128){\scriptsize {\em $M$}}
\put(97, -140){\scriptsize {\em $M$}}
\put(48, -173){\scriptsize {\em $M$}}
\put(20, -85){\line(2,1){30}}
\put(20, -85){\vector(-2,1){15}}
\put(20, -85){\vector(0,1){15}}
\put(80, -115){\makebox(0,0){$\bullet$}}
\put(80, -123){\scriptsize {\em $v_2$}}
\put(80, -115){\vector(0,1){20}}
\put(80, -115){\vector(1,0){15}}
\put(80, -115){\line(1,0){24}}
\put(80, -115){\line(-1,-2){12}}
\put(68, -138){\makebox(0,0){$\bullet$}}
\put(57, -135){\scriptsize {\em $v_3$}}
\put(32, -138){\makebox(0,0){$\bullet$}}
\put(68, -138){\vector(-1,0){23}}
\put(68, -138){\vector(2,-1){23}}
\put(23, -148){\scriptsize {\em $v_4$}}
\put(32, -138){\vector(-1,2){7}}
\put(32, -138){\line(0,-1){33}}
\put(32, -138){\vector(-2,-1){15}}
\put(32, -138){\line(1,0){35}}
\put(68, -138){\line(2,-1){33}}
\put(102, -154){\line(0,1){45}}
\put(102, -154){\vector(1,0){20}}
\put(102, -154){\makebox(0,0){$\bullet$}}
\put(96, -162){\scriptsize {\em $v_8$}}
\put(102, -154){\line(-2,-1){33}}
\put(102, -154){\vector(-2,-1){23}}
\put(102, -115){\makebox(0,0){$\bullet$}}
\put(103, -123){\scriptsize {\em $v_9$}}
\put(102, -115){\vector(1,0){20}}
\put(102, -115){\vector(0,1){20}}
\put(68, -170){\makebox(0,0){$\bullet$}}
\put(57, -165){\scriptsize {\em $v_{10}$}}
\put(68, -170){\vector(2,-1){12}}
\put(68, -170){\vector(0,-1){12}}
\put(32, -170){\makebox(0,0){$\bullet$}}
\put(35, -165){\scriptsize {\em $v_{11}$}}
\put(32, -170){\vector(-2,-1){15}}
\put(32, -170){\vector(0,1){18}}
\put(32, -170){\line(1,0){35}}


\put(180, -100){\makebox(0,0){$\bullet$}}
\put(185, -101){\scriptsize {\em $v_1$}}
\put(180, -100){\vector(2,-1){20}}
\put(180, -100){\line(-2,-1){30}}
\put(180, -100){\line(2,-1){30}}
\put(180, -100){\vector(0,1){20}}
\put(180, -100){\line(0,1){31}}
\put(180, -70){\makebox(0,0){$\bullet$}}
\put(185, -71){\scriptsize {\em $v_7$}}
\put(180, -70){\vector(-2,1){15}}
\put(180, -70){\vector(0,1){15}}
\put(150, -115){\makebox(0,0){$\bullet$}}
\put(143, -123){\scriptsize {\em $v_5$}}
\put(150, -115){\vector(-2,-1){15}}
\put(150, -115){\line(1,-2){12}}
\put(150, -115){\line(0,1){32}}
\put(150, -115){\vector(0,1){17}}
\put(150, -85){\makebox(0,0){$\bullet$}}
\put(140, -93){\scriptsize {\em $v_6$}}
\put(160, -80){\scriptsize {\em $M$}}
\put(160, -110){\scriptsize {\em $M$}}
\put(199, -128){\scriptsize {\em $M$}}
\put(227, -140){\scriptsize {\em $M$}}
\put(158, -161){\scriptsize {\em $M$}}
\put(150, -85){\line(2,1){30}}
\put(150, -85){\vector(-2,1){15}}
\put(150, -85){\vector(0,1){15}}
\put(210, -115){\makebox(0,0){$\bullet$}}
\put(210, -123){\scriptsize {\em $v_2$}}
\put(210, -115){\vector(0,1){20}}
\put(210, -115){\vector(1,0){15}}
\put(210, -115){\line(1,0){24}}
\put(210, -115){\line(-1,-2){12}}
\put(198, -138){\makebox(0,0){$\bullet$}}
\put(187, -135){\scriptsize {\em $v_3$}}
\put(162, -138){\makebox(0,0){$\bullet$}}
\put(198, -138){\vector(-1,0){23}}
\put(198, -138){\vector(2,-1){23}}
\put(153, -148){\scriptsize {\em $v_4$}}
\put(162, -138){\vector(-1,2){7}}
\put(162, -138){\line(0,-1){33}}
\put(162, -138){\vector(-2,-1){15}}
\put(162, -138){\line(1,0){35}}
\put(198, -138){\line(2,-1){33}}
\put(232, -154){\line(0,1){45}}
\put(232, -154){\vector(1,0){20}}
\put(232, -154){\makebox(0,0){$\bullet$}}
\put(230, -162){\scriptsize {\em $v_8$}}
\put(232, -154){\line(-2,-1){33}}
\put(232, -154){\vector(-2,-1){23}}
\put(232, -115){\makebox(0,0){$\bullet$}}
\put(233, -123){\scriptsize {\em $v_9$}}
\put(232, -115){\vector(1,0){20}}
\put(232, -115){\vector(0,1){20}}
\put(198, -170){\makebox(0,0){$\bullet$}}
\put(187, -165){\scriptsize {\em $v_{10}$}}
\put(198, -170){\vector(0,-1){12}}
\put(198, -170){\vector(-1,0){15}}
\put(162, -170){\makebox(0,0){$\bullet$}}
\put(165, -175){\scriptsize {\em $v_{11}$}}
\put(162, -170){\vector(0,-1){12}}
\put(162, -170){\vector(-1,0){12}}
\put(162, -170){\line(1,0){35}}

\put(58, -198){\scriptsize {\em $(1)$}}
\put(188, -198){\scriptsize {\em $(2)$}}
\put(318, -198){\scriptsize {\em $(3)$}}
\put(448, -198){\scriptsize {\em $(4)$}}
\put(58, -358){\scriptsize {\em $(5)$}}
\put(198, -358){\scriptsize {\em $(6)$}}
\put(318, -358){\scriptsize {\em $(7)$}}
\put(178, -378){\scriptsize { Fig. 1: Reducible configurations}}

\put(310, -100){\makebox(0,0){$\bullet$}}
\put(315, -101){\scriptsize {\em $v_1$}}
\put(310, -100){\vector(2,-1){20}}
\put(310, -100){\line(-2,-1){30}}
\put(310, -100){\line(2,-1){30}}
\put(310, -100){\vector(0,1){20}}
\put(310, -100){\line(0,1){31}}
\put(310, -70){\makebox(0,0){$\bullet$}}
\put(315, -71){\scriptsize {\em $v_7$}}
\put(310, -70){\vector(-2,1){15}}
\put(310, -70){\vector(0,1){15}}
\put(280, -115){\makebox(0,0){$\bullet$}}
\put(273, -123){\scriptsize {\em $v_5$}}
\put(280, -115){\vector(-2,-1){15}}
\put(280, -115){\line(1,-2){12}}
\put(280, -115){\line(0,1){32}}
\put(280, -115){\vector(0,1){17}}
\put(280, -85){\makebox(0,0){$\bullet$}}
\put(270, -93){\scriptsize {\em $v_6$}}
\put(290, -80){\scriptsize {\em $M$}}
\put(290, -110){\scriptsize {\em $M$}}
\put(349, -130){\scriptsize {\em $M$}}
\put(338, -173){\scriptsize {\em $M$}}
\put(286, -160){\scriptsize {\em $M$}}
\put(280, -85){\line(2,1){30}}
\put(280, -85){\vector(-2,1){15}}
\put(280, -85){\vector(0,1){15}}
\put(340, -115){\makebox(0,0){$\bullet$}}
\put(345, -118){\scriptsize {\em $v_2$}}
\put(340, -115){\vector(0,1){20}}
\put(340, -115){\vector(-1,-2){7}}
\put(340, -115){\line(1,-1){22}}
\put(340, -115){\line(-1,-2){12}}
\put(328, -138){\makebox(0,0){$\bullet$}}
\put(317, -135){\scriptsize {\em $v_3$}}
\put(328, -138){\line(0,-1){30}}
\put(292, -138){\makebox(0,0){$\bullet$}}
\put(328, -138){\vector(-1,0){23}}
\put(328, -138){\vector(0,-1){23}}
\put(283, -148){\scriptsize {\em $v_4$}}
\put(292, -138){\vector(-1,2){7}}
\put(292, -138){\line(0,-1){33}}
\put(292, -138){\vector(-2,-1){15}}
\put(292, -138){\line(1,0){35}}
\put(362, -170){\makebox(0,0){$\bullet$}}
\put(348, -167){\scriptsize {\em $v_{10}$}}
\put(362, -170){\vector(0,-1){12}}
\put(362, -170){\vector(1,0){12}}
\put(362, -138){\makebox(0,0){$\bullet$}}
\put(364, -144){\scriptsize {\em $v_{11}$}}
\put(362, -138){\line(0,-1){31}}
\put(362, -138){\vector(1,1){13}}
\put(362, -138){\vector(0,-1){20}}
\put(328, -170){\makebox(0,0){$\bullet$}}
\put(317, -167){\scriptsize {\em $v_9$}}
\put(328, -170){\vector(0,-1){12}}
\put(328, -170){\vector(-1,0){20}}
\put(328, -170){\line(1,0){33}}
\put(292, -170){\makebox(0,0){$\bullet$}}
\put(295, -167){\scriptsize {\em $v_8$}}
\put(292, -170){\vector(0,-1){12}}
\put(292, -170){\vector(-1,0){12}}
\put(292, -170){\line(1,0){35}}


\put(440, -100){\makebox(0,0){$\bullet$}}
\put(445, -101){\scriptsize {\em $v_1$}}
\put(440, -100){\vector(2,-1){20}}
\put(440, -100){\line(-2,-1){30}}
\put(440, -100){\line(2,-1){30}}
\put(440, -100){\vector(0,1){20}}
\put(440, -100){\line(0,1){31}}
\put(440, -70){\makebox(0,0){$\bullet$}}
\put(445, -71){\scriptsize {\em $v_7$}}
\put(440, -70){\vector(-2,1){15}}
\put(440, -70){\vector(0,1){15}}
\put(410, -115){\makebox(0,0){$\bullet$}}
\put(403, -123){\scriptsize {\em $v_5$}}
\put(410, -115){\vector(-2,-1){15}}
\put(410, -115){\line(1,-2){12}}
\put(410, -115){\line(0,1){32}}
\put(410, -115){\vector(0,1){17}}
\put(410, -85){\makebox(0,0){$\bullet$}}
\put(400, -93){\scriptsize {\em $v_6$}}
\put(420, -80){\scriptsize {\em $M$}}
\put(420, -110){\scriptsize {\em $M$}}
\put(479, -130){\scriptsize {\em $M$}}
\put(468, -173){\scriptsize {\em $M$}}
\put(416, -160){\scriptsize {\em $M$}}
\put(410, -85){\line(2,1){30}}
\put(410, -85){\vector(-2,1){15}}
\put(410, -85){\vector(0,1){15}}
\put(470, -115){\makebox(0,0){$\bullet$}}
\put(475, -118){\scriptsize {\em $v_2$}}
\put(470, -115){\vector(0,1){20}}
\put(470, -115){\vector(-1,-2){7}}
\put(470, -115){\line(1,-1){22}}
\put(470, -115){\line(-1,-2){12}}
\put(458, -138){\makebox(0,0){$\bullet$}}
\put(447, -135){\scriptsize {\em $v_3$}}
\put(458, -138){\line(0,-1){30}}
\put(422, -138){\makebox(0,0){$\bullet$}}
\put(458, -138){\vector(-1,0){23}}
\put(458, -138){\vector(0,-1){23}}
\put(413, -148){\scriptsize {\em $v_4$}}
\put(422, -138){\vector(-1,2){7}}
\put(422, -138){\line(0,-1){33}}
\put(422, -138){\vector(-2,-1){15}}
\put(422, -138){\line(1,0){35}}
\put(492, -170){\makebox(0,0){$\bullet$}}
\put(478, -167){\scriptsize {\em $v_{10}$}}
\put(492, -170){\vector(0,1){23}}
\put(492, -170){\vector(1,-1){12}}
\put(492, -138){\makebox(0,0){$\bullet$}}
\put(478, -144){\scriptsize {\em $v_{11}$}}
\put(492, -138){\line(0,-1){31}}
\put(492, -138){\vector(2,1){15}}
\put(492, -138){\vector(0,1){15}}
\put(458, -170){\makebox(0,0){$\bullet$}}
\put(447, -167){\scriptsize {\em $v_9$}}
\put(458, -170){\vector(0,-1){12}}
\put(458, -170){\vector(-1,0){23}}
\put(458, -170){\line(1,0){33}}
\put(422, -170){\makebox(0,0){$\bullet$}}
\put(425, -167){\scriptsize {\em $v_8$}}
\put(422, -170){\vector(-1,0){12}}
\put(422, -170){\vector(0,-1){12}}
\put(422, -170){\line(1,0){37}}


\put(50, -260){\makebox(0,0){$\bullet$}}
\put(38, -258){\scriptsize {\em $v_1$}}
\put(50, -260){\vector(-2,-1){20}}
\put(50, -260){\line(-2,-1){30}}
\put(50, -260){\line(2,-1){30}}
\put(50, -260){\vector(1,1){12}}
\put(50, -260){\line(1,1){21}}
\put(119, -237){\makebox(0,0){$\bullet$}}
\put(115, -245){\scriptsize {\em $v_6$}}
\put(119, -237){\vector(-2,1){15}}
\put(119, -237){\vector(0,1){15}}
\put(72, -237){\makebox(0,0){$\bullet$}}
\put(60, -235){\scriptsize {\em $v_7$}}
\qbezier(68, -321)(140, -275)(72, -237)
\qbezier(110, -305)(135, -245)(72, -237)
\put(72, -237){\line(1,0){45}}

\put(22, -305){\scriptsize {\em $M$}}
\put(58, -270){\scriptsize {\em $M$}}
\put(90, -240){\scriptsize {\em $M$}}
\put(20, -275){\makebox(0,0){$\bullet$}}
\put(13, -283){\scriptsize {\em $v_5$}}
\put(20, -275){\vector(-2,-1){15}}
\put(20, -275){\line(1,-4){12}}
\put(20, -275){\vector(-1,1){12}}
\put(80, -275){\makebox(0,0){$\bullet$}}
\put(68, -280){\scriptsize {\em $v_2$}}
\put(80, -275){\line(1,-1){30}}
\put(80, -275){\vector(1,-1){15}}
\put(80, -275){\vector(1,1){15}}
\put(80, -275){\line(1,1){37}}
\put(80, -275){\line(-1,-4){12}}
\put(105, -275){\vector(0,-1){5}}
\put(117, -275){\vector(0,-1){5}}
\put(68, -321){\makebox(0,0){$\bullet$}}
\put(57, -330){\scriptsize {\em $v_3$}}
\put(32, -321){\makebox(0,0){$\bullet$}}
\put(68, -321){\vector(-1,0){23}}
\put(68, -321){\vector(1,4){7}}
\put(25, -330){\scriptsize {\em $v_4$}}
\put(32, -321){\vector(-1,1){12}}
\put(32, -321){\vector(-2,-1){15}}
\put(32, -321){\line(1,0){35}}
\put(110, -305){\makebox(0,0){$\bullet$}}
\put(105, -315){\scriptsize {\em $v_9$}}
\put(110, -305){\vector(1,0){20}}
\put(110, -305){\vector(1,-1){14}}


\put(200, -260){\makebox(0,0){$\bullet$}}
\put(205, -261){\scriptsize {\em $v_1$}}
\put(200, -260){\vector(2,-1){20}}
\put(200, -260){\line(-2,-1){30}}
\put(200, -260){\line(2,-1){30}}
\put(200, -260){\vector(0,1){20}}
\put(200, -260){\line(0,1){31}}
\put(190, -249){\vector(-4,1){2}}

\put(182, -240){\scriptsize {\em $M$}}
\put(182, -270){\scriptsize {\em $M$}}
\put(193, -300){\scriptsize {\em $M$}}
\put(193, -332){\scriptsize {\em $M$}}
\put(200, -230){\makebox(0,0){$\bullet$}}
\put(205, -231){\scriptsize {\em $v_7$}}
\put(200, -230){\vector(-2,1){15}}
\put(200, -230){\vector(0,1){15}}
\put(170, -275){\makebox(0,0){$\bullet$}}
\put(158, -277){\scriptsize {\em $v_5$}}
\put(170, -275){\vector(1,-2){8}}
\put(170, -275){\line(1,-2){12}}
\put(170, -275){\line(0,1){32}}
\put(170, -245){\makebox(0,0){$\bullet$}}
\put(160, -253){\scriptsize {\em $v_6$}}
\put(170, -245){\line(2,1){30}}
\put(170, -245){\vector(-2,1){15}}
\put(170, -245){\vector(0,-1){15}}

\put(230, -275){\makebox(0,0){$\bullet$}}
\put(230, -283){\scriptsize {\em $v_2$}}
\put(230, -275){\vector(1,0){15}}
\put(230, -275){\line(-1,-2){12}}
\put(218, -298){\makebox(0,0){$\bullet$}}
\put(207, -295){\scriptsize {\em $v_3$}}
\put(182, -298){\makebox(0,0){$\bullet$}}
\put(218, -298){\vector(1,2){7}}
\put(218, -298){\vector(0,-1){23}}
\put(173, -308){\scriptsize {\em $v_4$}}
\put(152, -313){\vector(-1,-4){2}}
\put(182, -298){\vector(0,-1){20}}
\put(182, -298){\line(0,-1){33}}
\put(182, -298){\vector(-2,-1){15}}
\put(182, -298){\line(1,0){35}}
\put(218, -298){\line(0,-1){33}}

\put(218, -330){\makebox(0,0){$\bullet$}}
\put(207, -325){\scriptsize {\em $v_9$}}
\put(218, -330){\vector(1,-1){12}}
\put(182, -330){\makebox(0,0){$\bullet$}}
\put(185, -325){\scriptsize {\em $v_8$}}
\put(182, -330){\vector(-2,-1){15}}
\put(182, -330){\vector(-1,1){13}}
\put(182, -330){\line(1,0){35}}
\qbezier(218, -330)(113, -368)(170, -275)
\qbezier(230, -275)(250, -260)(170, -245)


\put(310, -260){\makebox(0,0){$\bullet$}}
\put(315, -261){\scriptsize {\em $v_1$}}
\put(310, -260){\vector(-2,-1){20}}
\put(310, -260){\line(-2,-1){30}}
\put(310, -260){\line(2,-1){30}}
\put(310, -260){\vector(0,1){20}}
\put(310, -260){\line(0,1){31}}

\put(320, -241){\scriptsize {\em $M$}}
\put(282, -291){\scriptsize {\em $M$}}
\put(329, -291){\scriptsize {\em $M$}}
\put(302, -335){\scriptsize {\em $M$}}
\put(358, -317){\scriptsize {\em $M$}}

\put(310, -230){\makebox(0,0){$\bullet$}}
\put(315, -231){\scriptsize {\em $v_7$}}
\put(310, -230){\vector(-2,1){15}}
\put(310, -230){\vector(0,1){15}}
\put(280, -275){\makebox(0,0){$\bullet$}}
\put(273, -283){\scriptsize {\em $v_5$}}
\put(280, -275){\vector(-2,-1){15}}
\put(280, -275){\line(1,-2){12}}
\put(280, -275){\vector(0,1){17}}
\put(340, -245){\makebox(0,0){$\bullet$}}
\put(345, -250){\scriptsize {\em $v_6$}}
\put(340, -245){\line(-2,1){28}}
\put(340, -245){\vector(2,1){15}}
\put(340, -245){\vector(0,1){15}}
\put(340, -275){\makebox(0,0){$\bullet$}}
\put(345, -280){\scriptsize {\em $v_2$}}
\put(340, -275){\vector(0,1){20}}
\put(340, -275){\vector(-2,1){17}}
\put(340, -275){\line(0,1){33}}
\put(340, -275){\line(1,-1){22}}
\put(340, -275){\line(-1,-2){12}}
\put(328, -298){\makebox(0,0){$\bullet$}}
\put(317, -295){\scriptsize {\em $v_3$}}
\put(328, -298){\line(0,-1){30}}
\put(292, -298){\makebox(0,0){$\bullet$}}
\put(328, -298){\vector(-1,0){23}}
\put(328, -298){\vector(0,-1){23}}
\put(283, -308){\scriptsize {\em $v_4$}}
\put(292, -298){\vector(0,-1){19}}
\put(292, -298){\line(0,-1){33}}
\put(292, -298){\vector(-2,-1){15}}
\put(292, -298){\line(1,0){35}}
\put(362, -330){\makebox(0,0){$\bullet$}}
\put(346, -328){\scriptsize {\em $v_{10}$}}
\put(362, -330){\vector(1,0){12}}
\put(362, -330){\vector(0,-1){12}}
\put(362, -298){\makebox(0,0){$\bullet$}}
\put(364, -304){\scriptsize {\em $v_{11}$}}
\put(362, -298){\line(0,-1){31}}
\put(362, -298){\vector(1,1){13}}
\put(362, -298){\vector(-1,1){13}}
\put(328, -330){\makebox(0,0){$\bullet$}}
\put(317, -328){\scriptsize {\em $v_9$}}
\put(328, -330){\vector(0,-1){12}}
\put(328, -330){\vector(1,0){18}}
\put(328, -330){\line(1,0){33}}
\put(292, -330){\makebox(0,0){$\bullet$}}
\put(293, -328){\scriptsize {\em $v_8$}}
\put(292, -330){\vector(0,-1){12}}
\put(292, -330){\vector(-1,0){12}}
\put(292, -330){\line(1,0){35}}
\end{picture}

 \vskip 12cm

\begin{lemma}\label{lem6} A NC-graph $G\in {\cal H}$ has no
 a $(3,4,3,4)$-face,
\end{lemma}
\begin{proof}
Suppose otherwise that $G$ has a $(3,4,3,4)$-face $[v_1v_2v_3v_4]$. Let $N(v_1)=\{v_{11},v_2, v_4\}, N(v_2)\\=\{v_{21}, v_{22}, v_1, v_3\}, N(v_3)=\{v_{31}, v_2, v_4\}$ and $N(v_4)=\{v_{41}, v_{42}, v_1, v_3\}$.  By the minimality of $G$, there is a $(2,1)$-decomposition $(D^*, M^*)$ of
$G-\{v_1, v_2, v_3, v_4\}$.
  Let $M=M^*\cup \{v_1v_2, v_3v_4\}$ and $D=D^*\cup\{\overrightarrow{v_1v_{11}}, \overrightarrow{v_2v_{21}}, \overrightarrow{v_2v_{22}}, \overrightarrow{v_3v_{31}}, \overrightarrow{v_4v_{41}}, \overrightarrow{v_4v_{42}}, \overrightarrow{v_1v_4}, \overrightarrow{v_3v_2}\}$. Then $(D,M)$ is a $(2,1)$-decomposition of $G$, a contradiction.
\end{proof}

\section{Proofs of Theorem \ref{th0} and \ref{th1}}

We are now ready to  complete the proof of Theorem~\ref{th0} and \ref{th1}.  We define initial charge  $\mu(x)=d(x)-4$ for each $x\in V\cup F$. By Euler's Formula $|V(G)|+|F(G)|-|E(G)|\geq 0$,
\[
\sum_{v\in V(G)}(d(v)-4)+\sum_{f\in F(G)}(d(f)-4)\leq 0.
\]
Let $\mu'(x)$ be the charge of $x\in V(G)\cup F(G)$ after the discharge procedure. In order to prove the Theorems \ref{th0} and \ref{th1}, we shall design some discharging rules so that after discharging. Since the total sum of weights is kept unchanged,  the new weight function $\mu'$ satisfies

(I) $\mu'(x)\geq0$ for all $x\in V(G)\cup F(G)$;

(II) There exists some $x^*\in V(G)\cup F(G)$ such that $\mu'(x^*)>0$.

Thus
\[
0<\sum_{x\in V(G)\cup F(G)}\mu'(x)= \sum_{x\in V(G)\cup F(G)}\mu(x)=0.
\]
This contradiction completes our proofs.

\subsection{Proofs of Theorem~\ref{th0}(1) and (3).}

In this section, we prove Theorem~\ref{th0}(1) and (3).
Now we define the  discharge rules as follows.
\medskip

\begin{enumerate}[(R1)]
\item Every $5$-vertex sends $\frac{1}{3}$ to each incident $(5^+,5^+,5^+)$-face, $\frac{1}{2}$ to each incident $(4,5,5)$-face and $\frac{5}{12}$ to each incident $(4,5,6^+)$-face.

\item Every $6^+$-vertex sends $\frac{7}{12}$ to each incident 3-face.

\item Every $5^+$-face sends $\frac{11}{60}$ to each incident vertex.
\end{enumerate}

\medskip

It suffices to show that the new weight function $\mu'$ satisfies Properties (I) and (II).

We first check $\mu'(v)\geq0$ for all $v\in V(G)$. By Lemma \ref{lem1} (1), $d(v)\geq4$.
\begin{enumerate}
\item $d(v)=4$. Since no 4-vertex is involved in the discharge procedure, $\mu'(v)=\mu(v)=4-4=0$.

 \item $d(v)=5$. Then $\mu(v)=1$. By Lemma \ref{lem3}, $n_3(v)\le3$. If $n_3(v)\le2$, then $v$ is incident with at most one $(4,5,5)$-face  by Lemma \ref{lem2}(1) and is not incident with any $(4, 4, 5^-)$-face by Lemma \ref{lem1}(2). By (R1),  $\mu'(v)\geq1-\frac{1}{2}-\frac{5}{12}=\frac{1}{12}>0$. Let $n_3(v)=3$. Then $v$ is incident with two $4^+$-faces. If $v$ is incident with one 4-face, then $G$ has a chord 5-face and so does a chord 7-face and adjacent 4-cycles, contrary to our assumption. Thus, $n_4(v)=0$.   This implies that $n_{5^+}(v)=2$. By Lemmas~\ref{lem2}(1) and \ref{lem1}(2), $v$ is incident with at most one $(4,5,5)$-face and is not incident with any $(4, 4, 5^-)$-face.  Thus $\mu'(v)\geq1-\frac{1}{2}-2\times\frac{5}{12}+2\times\frac{11}{60}=\frac{1}{30}>0$ by (R1) and (R3).

 \item $d(v)=6$. Then $\mu(v)=2$. By Lemma \ref{lem3}, $n_3(v)\le 4$. If $n_3(v)\le3$, then $\mu'(v)\geq2-3\times\frac{7}{12}=\frac{1}{4}>0$ by (R2). Thus, assume that $n_3(v)=4$. In this case, $v$ is incident with two $4^+$-faces. If $v$ is indeed incident one 4-face, then $G$ has a chord 5-cycle and so does a chord 7-cycle and adjacent 4-cycles,  contrary to our assumption. Thus, $n_4(v)=0$. This implies that $n_{5^+}(v)=2$. Thus $\mu'(v)\geq2-4\times\frac{7}{12}+2\times\frac{11}{60}=\frac{1}{30}>0$ by (R2) and (R3).

 \item $d(v)\geq7$. By Lemma \ref{lem3}, $n_3(v)\le\lfloor\frac{2d(v)}{3}\rfloor$. Thus $\mu'(v)\geq d(v)-4-\frac{7}{12}\times\lfloor\frac{2d(v)}{3}\rfloor\geq d(v)-4-\frac{7}{12}\times\frac{2d(v)}{3}=\frac{11}{18}d(v)-4\geq \frac{5}{18}>0$ by (R2).
\end{enumerate}
Then we check $\mu'(f)\geq0$ for all $f\in F(G)$.
\begin{enumerate}
 \item $d(f)=3$.  Then $\mu(f)=-1$. By Lemma \ref{lem1}(2), $v$ is not incident with any $(4, 4, 5^-)$-face. If $f$ is a $(4,5,5)$-face, then $\mu'(f)\geq-1+2\times\frac{1}{2}=0$ by (R1). If $f$ is a $(4,5^+,6^+)$-face, then $\mu'(f)\geq-1+\frac{5}{12}+\frac{7}{12}=0$ by (R1) and (R2). If $f$ is a $(5^+,5^+,5^+)$-face, then $\mu'(f)\geq-1+3\times\frac{1}{3}=0$ by (R1) and (R2).

\item  $d(f)=4$. Since no 4-face is involved in the discharge procedure, $\mu(f)=\mu'(f)=4-4=0$.

\item  $d(f)\geq5$. Then $\mu'(f)\geq d(f)-4-\frac{11}{60}d(f)=\frac{49}{60}d(f)-4\geq\frac{1}{12}>0$ by (R3).
\end{enumerate}
So far, we have proved Property (I). Assume that Property (II) does not hold. This implies that $\mu'(x)=0$ for all $x\in V(G)\cup F(G)$. We observe the above proof and  have each of the following holds.

(a) For each vertex $v\in V(G)$, $d(v)=4$;

(b) For each face $f\in F (G)$, $3\le d(f)\le 4$.\\
By (a), $G$ has no $5^+$-vertices. Thus $G$ is $4$-regular, which is contrary to  Lemma \ref{lem1} (2). This completes the proofs of  Theorem \ref{th0} (1) and (3).

\subsection{Proof of Theorem~\ref{th0}(2)}

In this section, we prove  Theorem~\ref{th0}(2).
Now we define the  discharge rules as follows.
\medskip

\begin{enumerate}[(R1)]
\item  Every 5-vertex sends $\frac{1}{3}$ to each incident $(5^+, 5^+, 5^+)$-face, $\frac{1}{2}$ to each incident $(4,5,5)$-face, $\frac{5}{12}$ to each incident $(4,5,6)$-face and $\frac{61}{150}$ to each incident $(4,5,7^+)$-face.

\item Every $6$-vertex sends $\frac{7}{12}$ to each incident 3-face.

\item Every $7^+$-vertex sends $\frac{89}{150}$ to each incident 3-face.

\item Every $5$-face sends $\frac{11}{60}$ to each incident vertex.

\item Every $6^+$-face sends  $\frac{49}{150}$ to each incident vertex.
\end{enumerate}

\medskip

It suffices to show that the new weight function $\mu'$ satisfies Properties (I) and (II). Note that each $3$-face is not adjacent to $5$-face since $G$ has no chord $6$-cycles.

We first check $\mu'(v)\geq0$ for all $v\in V(G)$. By Lemma \ref{lem1} (1), $d(v)\geq4$.
\begin{enumerate}
\item $d(v)=4$. Since no 4-vertex is involved in the discharge procedure, $\mu'(v)=\mu(v)=4-4=0$.

\item  $d(v)=5$.  By Lemma \ref{lem4}, $n_3(v)\le3$. If $n_3(v)\le2$, then $v$ is incident with at most one $(4,5,5)$-face  by Lemma \ref{lem2}(1) and is not incident with any $(4, 4, 5^-)$-face by Lemma \ref{lem1}(2). By (R1), $\mu'(v)\geq1-\frac{1}{2}-\frac{5}{12}=\frac{1}{12}>0$ by (R1). Thus, assume that $n_3(v)=3$.

Suppose that $v_1, v_2, \ldots, v_5$ are the neighbors of $v$ in clockwise order, and $f_1,  f_2, \ldots , f_5$ are the incident faces of $v$ with $vv_i, vv_{i+1}\in b(f_i)$ for $i=1, 2, \ldots, 5$ where indices are taken modulo 5. By symmetry, there are two cases: either  $f_1,f_2$ and $f_3$ or $f_1, f_2$ and $f_4$ are 3-faces.

In the former case, since $G$ has no chord $6$-cycles, each of $f_4$ and $f_5$ is not a $5$-face. Thus, $n_5(v)=0$. We claim that  at most one of $f_4$ and $f_5$ is a 4-face. Suppose otherwise. Let $f_4=[vv_4xv_5]$ and $f_5=[v_1vv_5y]$. Since $G$ has no chord 6-cycle, $x, y\in \{v, v_1, v_2, v_3, v_4, v_5\}$. Since $G$ is a simple graph, $x\notin \{v, v_2, v_3\}$. Since $n_3(v)=3$, $x\notin \{v_1,v_4,v_5\}$. Similarly, $y\notin\{v, v_5, v_2, v_4, v_1\}$. Thus, $x=v_2$ and $y=v_3$. In this case, $G$ has a chord 6-cycle $vv_4v_3v_1v_2v_5v$, a contradiction.
Thus $n_4(v)\le1$. This implies that $1\le n_{6^+}(v)\le2$. By Lemma \ref{lem2}(1), $v$ is incident with at most one $(4,5,5)$-face. If $v$ is not incident with $(4,5,5)$-face, then $\mu'(v)\geq1-3\times\frac{5}{12}+\frac{49}{150}=\frac{23}{300}>0$ by (R1) and (R5). Thus, assume that $v$ is incident with one $(4,5,5)$-face. By Lemma \ref{lem2}(2), $v$ is  incident with at most one $(4,5,6)$-faces. By (R1) and (R5),
$\mu'(v)\geq1-\frac{1}{2}-\frac{5}{12}-\frac{61}{150}+\frac{49}{150}=\frac{1}{300}>0$.

In the latter case, since $G$ has no chord $6$-cycles,  none of $f_3$ and $f_5$ is a $5$-face. Thus $n_5(v)=0$.
If $f_3=[vv_3xv_4]$ is a $4$-face, then $x\notin\{v, v_2, v_3, v_4, v_5\}$ since $G$ is a simple graph and by Lemma~\ref{lem1}(1). If $x\neq v_1$, then  $vv_2v_3xv_4v_5v$ is a 6-cycle with a chord $vv_3$, a contradiction. If $x=v_1$, then $v_1v_4v_5vv_3v_2v_1$ is a $6$-cycle with a chord $vv_4$, a contradiction. By symmetry, $f_5$ is not a $4$-face. Thus, $n_4(v)=0$.
 This implies that $n_{6^+}(v)=2$. Thus $\mu'(v)\geq 1-3\times\frac{1}{2}+2\times\frac{49}{150}=\frac{23}{150}>0$ by (R1) and (R5).

\item $d(v)=6$. Then $\mu(v)=2$. By Lemma \ref{lem4}, $n_3(v)\le4$. If $n_3(v)\le3$, then $\mu'(v)\geq2-3\times\frac{7}{12}=\frac{1}{4}>0$ by (R2). Thus, assume that $n_3(v)=4$.

We now prove $n_4(v)=n_5(v)=0$.
 Assume that $v_1, v_2, \ldots, v_6$ are the neighbors of $v$ in clockwise order, and $f_1,  f_2, \ldots , f_6$ are the incident faces of $v$ with $vv_i, vv_{i+1}\in b(f_i)$ for $i=1, 2, \ldots, 6$ where indices are taken modulo 6. Since $G$ has no chord $6$-cycles, $n_5(v)=0$ by Lemma~\ref{lem4}. By Lemma~\ref{lem4} and symmetry, we consider two cases: either  $f_1, f_2, f_3, f_5$ or $f_1, f_2, f_4,f_5$ are four 3-faces.

In the former case,
assume that $f_4=[vv_4xv_5]$ is a $4$-face. Since $G$ is a simple graph, $x\notin\{v, v_3, v_4, v_5, v_6\}$ by Lemma~\ref{lem1}(1).
If $x=v_2$, then $v_2v_4v_3vv_6v_5v_2$ is a $6$-cycle with a chord $vv_4$, a contradiction. If $x=v_1$, then $v_1v_2v_3v_4vv_5v_1$ is a $6$-cycle with a chord $vv_3$, a contradiction. If $x\neq v_1$ and $x\neq v_2$, then $xv_5vv_2v_3v_4x$ is a $6$-cycle with a chord $vv_4$, a contradiction. Thus $f_4$ is not a $4$-face. By symmetry, $f_6$ is not a $4$-face. Thus $n_4(v)=0$.

In the latter case,
assume that $f_3=[vv_3xv_4]$ is a $4$-face.  Since $G$ is a simple graph, $x\notin\{v, v_2, v_3, v_4, v_5\}$ by Lemma~\ref{lem1}(1).
If $x=v_1$, then $v_1v_2vv_6v_5v_4v_1$ is a $6$-cycle with a chord $vv_4$, a contradiction. By symmetry, $x\neq v_6$.  If $x\neq v_1$ and $x\neq v_6$, then $v_1v_2v_3xv_4v_5v_1$ is a 6-cycle with a chord $vv_5$, a contradiction. Thus $f_3$ is not a $4$-face. By symmetry, $f_6$ is not a $4$-face. Thus, $n_4(v)=0$.

So far, we have proved that $n_4(v)=n_5(v)=0$. This implies that $n_{6^+}(v)=2$. Thus $\mu'(v)\geq2-4\times\frac{7}{12}+2\times\frac{49}{150}=\frac{8}{25}>0$ by (R2) and (R5).

\item $d(v)\geq7$, then by Lemma \ref{lem4}, $v$ is incident with at most $(d(v)-2)$ 3-faces. Thus $\mu'(v)\geq d(v)-4-\frac{89}{150}(d(v)-2)=\frac{61}{150}d(v)-\frac{422}{150}\geq\frac{1}{30}>0$ by (R3).
\end{enumerate}
Then we check $\mu'(f)\geq0$ for all $f\in F(G)$.
\begin{enumerate}
\item $d(f)=3$.  By Lemma \ref{lem1}(2), $v$ is not incident with any $(4, 4, 4^+)$-face.  If $f$ is a $(5^+, 5^+,5^+)$-face, then $\mu'(f)\geq-1+3\times\frac{1}{3}=0$ by (R1)--(R3). If $f$ is a $(4,5,5)$-face, then $\mu'(f)\geq-1+2\times\frac{1}{2}=0$ by (R1).  If $f$ is a $(4,5^+,6)$-face, then $\mu'(f)\geq-1+\frac{5}{12}+\frac{7}{12}=0$ by (R1)--(R3). If $f$ is a $(4,5^+,7^+)$-face, then $\mu'(f)\geq-1+\frac{61}{150}+\frac{89}{150}=0$ by (R1)--(R3).

\item $d(f)=4$. Since 4-faces are not involved in discharge procedure, $\mu(f)=\mu'(f)=0$.

\item $d(f)=5$. Then $\mu(f)=1$.  By (R4), $\mu'(f)\geq1-5\times\frac{11}{60}=\frac{5}{60}>0$.

\item  $d(f)\geq6$. By (R5), $\mu'(f)\geq d(f)-4-\frac{49}{150}d(f)=\frac{101}{150}d(f)-4\geq\frac{6}{150}>0$.
\end{enumerate}
We have proved Property (I). Assume that Property (II) does not hold. This implies that $\mu'(x)=0$ for all $x\in V(G)\cup F(G)$. We check above proof and obtain the following assertions.

(a) For each vertex $v\in V(G)$, $d(v)=4$;

(b) For each face $f\in F (G)$, $3\le d(f)\le 4$.\\
By (a), $G$ has no $5^+$-vertices. Thus $G$ is $4$-regular, which is contrary to  Lemma \ref{lem1} (2). This completes the proof of  Theorem \ref{th0} (2).

\subsection{Proof of Theorem~\ref{th1}(1)}

In this section, we prove Theorem~\ref{th1}(1).
Now we define the  discharge rules as follows.
\medskip

\begin{enumerate}[(R1)]
\item Every $5^+$-face sends $\frac{1}{3}$ to each incident $3$-vertex.
\end{enumerate}

\medskip

It suffices to show that the new weight function $\mu'$ satisfies Properties (I) and (II).

We first check $\mu'(v)\geq0$ for all $v\in V(G)$. By Lemma \ref{lem5} (1), $d(v)\geq3$.

\begin{enumerate}
\item $d(v)=3$. Then $\mu(v)=3-4=-1$. Since $G$ has no $3$- and $4$-cycles, $v$ is incident with three $5^+$-faces. Thus $\mu'(v)\geq -1+3\times\frac{1}{3}=0$ by (R1).

\item $d(v)=4$. Then $\mu'(v)=\mu(v)=4-4=0$.

\item $d(v)=5$. Then $\mu'(v)=\mu(v)=d(v)-4\geq1>0$.
\end{enumerate}

We further check $\mu'(f)\geq0$ for all $f\in F(G)$. Note that $d(f)\geq5$.

By Lemma \ref{lem5} (2), $f$ is incident with at most $\lfloor\frac{d(f)}{2}\rfloor$ $3$-vertices. Thus $\mu'(f)\geq d(f)-4-\frac{1}{3}\times\lfloor\frac{d(f)}{2}\rfloor\geq \frac{5}{6}d(f)-4\geq\frac{1}{6}>0$.

We have proved Property (I). Assume that Property (II) does not hold. This implies that $\mu'(x)=0$ for all $x\in V(G)\cup F(G)$. Considering above proof, we obtain that
 $G$ has no $5^+$-face and hence every face of $G$ is a $4^-$-face, contrary to our assumption that $G$ has no 3-cycle nor 4-cycle.
 This completes the proof of  Theorem \ref{th1}(1).

\subsection{Proof of Theorem~\ref{th1}(2)}

In this section, we prove Theorem~\ref{th1}(2). Since $G$ has no 6-cycle, each 3-vertex is  incident with at most one 4-face by Lemma~\ref{lem0}. A $3$-vertex $v$ is {\em bad} if $v$ is incident with one $4$-face and {\em good} otherwise.

Now we define the  discharge rules as follows.
\medskip

\begin{enumerate}[(R1)]
\item Every $5^+$-face sends $\frac{1}{3}$ to each incident good $3$-vertex and $\frac{1}{2}$ to each incident bad $3$-vertex.
\end{enumerate}

\medskip

It suffices to show that the new weight function $\mu'$ satisfies Properties (I) and (II). Note that each $4$-face is not adjacent to $4$-face by Lemma \ref{lem0} and \ref{lem5} (1).

We first check $\mu'(v)\geq0$ for all $v\in V(G)$. By Lemma~\ref{lem5}(1), $d(v)\geq3$.
\begin{enumerate}
\item $d(v)=3$.  If $v$ is good, then $v$ is incident with three $5^+$-faces. Thus
 $\mu'(v)\geq -1+3\times\frac{1}{3}=0$ by (R1). If $v$ is bad, then  $v$ is incident with two $5^+$-faces. Thus $\mu'(v)\geq-1+2\times\frac{1}{2}=0$ by (R1).

\item $d(v)=4$.  Since any 4-vertex does not involved in discharge procedure,  $\mu'(v)=\mu(v)=4-4=0$.

\item $d(v)\geq5$. Then $\mu(v)=d(v)-4$. Since any 5-vertex does not involved in discharge procedure,  $\mu'(v)=\mu(v)\geq1>0$.
\end{enumerate}
We further check $\mu'(f)\geq0$ for all $f\in F(G)$. Note that $d(f)\geq4$ and $d(f)\neq6$.
\begin{enumerate}
\item $d(f)=4$.  Since any 4-face does not involved in discharge procedure, $\mu'(f)=\mu(f)=4-4=0$.

\item $d(f)=5$. Then $\mu(f)=5-4=1$.
By Lemma \ref{lem5}(2), $f$ is incident with at most two $3$-vertices. If $v$ is incident with at most one $3$-vertex, then $\mu'(f)\geq1-\frac{1}{2}=\frac{1}{2}>0$ by (R1). Let $v$ be incident with two $3$-vertices $v_1$ and $v_2$. If one of $v_1$ and $v_2$  is  bad  and the other is good, then $\mu'(v)\geq1-\frac{1}{2}-\frac{1}{3}=\frac{1}{6}>0$ by (R1). If both $v_1$ and $v_2$ are bad $3$-vertices, then $\mu'(v)\geq1-2\times\frac{1}{2}=0$ by (R1).

 \item $d(f)\geq7$. By Lemma \ref{lem5}(2), $f$ is incident with at most $\lfloor\frac{d(f)}{2}\rfloor$ $3$-vertices. Thus $\mu'(v)\geq d(f)-4-\frac{1}{2}\times\lfloor\frac{d(f)}{2}\rfloor\geq \frac{3}{4}d(f)-4\geq\frac{5}{4}>0$.
\end{enumerate}
 We have proved Property (I). Assume that Property (II) does not hold. This implies that $\mu'(x)=0$ for all $x\in V(G)\cup F(G)$. Considering above proof, we establish the following claims.

\n{\bf Claim 1.} Each of the following holds.

(1) For each vertex $v\in V(G)$, $3\le d(v)\le 4$.

(2) For each face $f\in F(G)$, $f$ is either a $5$-face incident with two bad 3-vertices or a $4$-face.

\medskip
\n{\bf Claim 2.} Let $f$ be a 5-face. Then each of the following holds.

(1) $f$ is
not incident with three consecutively adjacent $4$-vertices.

(2) If $f$ is adjacent to a 4-face $g$, then $g$ is a $(3, 4, 4, 4)$-face.

\n{\bf Proof of Claim 2.} (1) By Claim 1(2), $f$ is incident two bad 3-vertices. By Lemma \ref{lem5}(2), $f$ is a $(4,3,4,3,4)$-face and hence $f$ is
not incident with three consecutively adjacent $4$-vertices.

(2) it follows by Lemma \ref{lem6}. $\blacksquare$

\medskip

  \n{\bf Claim 3.}
   $G$ has a 5-face.

\n{\bf Proof of Claim 3.}
Suppose otherwise  that $G$ has no $5$-face. By Claim 1(2), $G$ has only 4-faces. If $G$ has more than one 4-face, then $G$ contains two adjacent 4-faces, contrary to  Lemma \ref{lem0}. $\blacksquare$

\medskip

By Claims 1(2) and 3, we assume that $G$ has a 5-face $f=[v_1v_2v_3v_4v_5]$ incident with two bad 3-vertices. By Claim 2(1) and by symmetry, assume that $v_1$ and $v_3$ are two $3$-vertices. Thus, $v_1$ and $v_3$ are incident with one $4$-face and two $5$-faces. Assume that $v_1$ is incident with $f, f_1, f_2$ and $v_3$ is incident with $f, f_3, f_4$ where $f_1,f_3$ are two $4$-faces and $f_2, f_4$ are two $5$-faces. By Lemma~\ref{lem6}, $f_1$ and $f_3$ are two $(3,4,4,4)$-faces. We observe $f_1$ and consider  two following cases.

\n{\bf Case 1.} $v_1v_2\in b(f_1)$.

Let $f_1=[v_1v_2v_6v_7]$. We first claim that $v_6, v_7\notin \{v_1, \ldots, v_5\}$.
Note that  $v_6\notin \{v_1, v_2,v_3\}$ and $v_7\notin \{v_1, v_2, v_5\}$.  Since $G$ has no $3$-cycle, $v_6\notin\{v_4, v_5\}$ and $v_7\notin\{v_3, v_4\}$.
In this case, we consider two cases for $f_3$: either $v_2v_3\in b(f_3)$ or $v_3v_4\in b(f_3)$.

In the former case, let $f_3=[v_2v_3v_8v_9]$. We claim that $v_8, v_9\notin \{v_1, \ldots, v_7\}$.
Since $d(v_1)=3$, $v_8\neq v_1$ and $v_9\neq v_1$. Clearly, $v_8\notin \{ v_2, v_3, v_4\}$ and $v_9\notin \{v_2, v_3, v_6\}$ .  Since $G$ has no $3$-cycle,  $v_8\notin\{v_5, v_6\}$ and $v_9\notin\{v_4, v_5, v_7\}$.
 If $v_8=v_7$, then $G$ contains Configuration (5) in Fig.1, contrary to Lemma \ref{lem7}. Thus,
$v_8, v_9\notin \{v_1, \ldots, v_7\}$.

Let $v_2$ be incident with $f, f_1, f_5, f_3$ in clockwise order. By Lemma \ref{lem0}, $f_5$ is a $5$-face. By Claim 2(2), $d(v_2)=d(v_6)=d(v_9)=4$. Therefore, $f_5$ is incident with three consecutively adjacent $4$-vertices which is contrary to Claim 2(1).

In the latter case, let $f_3=[v_4v_3v_9v_8]$. We first claim that $v_8, v_9\notin\{v_1, \ldots, v_7\}$.
 Recall that $v_1$ is a $3$-vertex, $v_8\neq v_1$ and $v_9\neq v_1$. Obviously, $v_8\notin\{v_3, v_4, v_5\}$ and $v_9\notin\{v_2, v_3, v_4\}$ since $G$ is simple. Since $G$ has no $3$-cycle, $v_8\neq v_2$ and $v_9\notin\{v_5, v_6\}$.
If $v_8=v_6$, then $v_4v_3v_2v_1v_7v_6v_4$ is a $6$-cycle, a contradiction. Thus $v_8, v_9\notin\{v_1, \ldots, v_6\}$.
If $v_8=v_7$, then $v_9\notin\{v_1, \ldots, v_7\}$ and $v_7v_1v_5v_4v_3v_9v_7$ is a $6$-cycle, a contradiction. If $v_9=v_7$, then $v_8\notin\{v_1, \ldots, v_7\}$ and $v_7v_6v_2v_3v_4v_8v_7$ is a $6$-cycle, a contradiction.
Thus, $v_8, v_9\notin\{v_1,\ldots, v_7\}$.

 Let $v_2$ be incident with $f, f_1, f_5, f_4$ in clockwise order. By Lemma \ref{lem0}, $f_4$ is a $5$-face.
 Let $f_4=[v_2v_3v_9v_{10}v_{11}]$. We first assume that $v_{10}, v_{11}\notin\{v_1, \ldots, v_9\}$. In this case, $d(v_2)=d(v_9)=4$ and $d(v_3)=4$. By Claim 1(2), only one vertex in $\{v_{10}, v_{11}\}$ is a $3$-vertex.
Since $G$ does not contain Configuration (7) of Fig.1 in  Lemma \ref{lem7},  $d(v_{11})=4$ and $d(v_{10})=3$. By Lemma \ref{lem0} and
Claim 1 (2), $f_5$ is a 5-face. Moreover, $f_5$ is incident with three consecutively adjacent $4$-vertices $v_6, v_2$ and $v_{11}$, contrary to Claim 2 (1). Thus, assume that  $v_{10}\in\{v_1, \ldots, v_9\}$ or $v_{11}\in\{v_1, \ldots, v_9\}$.

Since $v_1$ is a $3$-vertex, $v_{10}\neq v_1$ and $v_{11}\neq v_1$.  Obviously, $v_{10}\notin\{v_2, v_3,v_8, v_9\}$ and $v_{11}\notin\{v_2, v_3, v_6\}$ since $G$ is simple.
Since $G$ has no $3$-cycle, $v_{10}\neq v_4$ and $v_{11}\notin\{v_4, v_5, v_7, v_9\}$.
 If $v_{10}=v_6$, then $v_1v_2v_3v_9v_6v_7v_1$ is a $6$-cycle, a contradiction. If $v_{10}=v_7$, then $v_7v_9v_3v_4v_5v_1v_7$ is a $6$-cycle, a contradiction. Thus $v_{10}\notin\{v_1, \ldots, v_4, v_6,\ldots, v_9\}$
and $v_{11}\notin\{v_1,\ldots, v_7,v_9\}$. Let $v_{10}=v_5$.
 If $v_{11}=v_8$, then $v_5v_4v_8v_5$ is a $3$-cycle, a contradiction. If $v_{11}\notin\{v_1, \ldots, v_9\}$, then $v_1v_5v_{11}v_2v_6v_7v_1$ is a $6$-cycle, a contradiction. If $v_{10}\notin\{v_1, \ldots, v_9\}$, then $v_{11}=v_8$ and $v_8v_9v_{10}v_8$ is a $3$-cycle, a contradiction.

\n{\bf Case 2.} $v_1v_5\in b(f_1)$.

Let $f_1=[v_1v_5v_6v_7]$. We first claim that $v_6, v_7\notin\{v_1, \ldots, v_5\}$.
Obviously, $v_6\notin\{v_1, v_4, v_5\}$ and $v_7\notin\{v_1, v_2, v_5\}$. Since $G$ has no $3$-cycle, $v_6\notin\{v_2, v_3\}$ and $v_7\notin\{v_3, v_4\}$.
In this case, we consider two cases of $f_3$: either $v_2v_3\in b(f_3)$ or $v_3v_4\in b(f_3)$.

In the former case, let $f_3=[v_2v_3v_8v_9]$. We claim that $v_8, v_9\notin\{ v_1, \ldots, v_7\}$. Since $v_1$ is a $3$-vertex, $v_8\neq v_1$ and $v_9\neq v_1$.
 Obviously, $v_8\notin\{ v_2, v_3, v_4\}$ and $v_9\notin\{v_2, v_3\}$ since $G$ is simple. Since $G$ has no $3$-cycle, $v_8\neq v_5$ and $v_9\notin\{v_4, v_5, v_7\}$.
If $v_8=v_7$, then $v_7v_6v_5v_1v_2v_3v_7$ is a $6$-cycle, a contradiction.
 Thus $v_8, v_9\notin\{v_1, \ldots, v_5, v_7\}$.
 If $v_8=v_6$, then $v_9\not=v_6$. In this case, $v_6v_5v_4v_3v_2v_9v_6$ is a $6$-cycle, a contradiction. If $v_9=v_6$, then $v_8\not=v_6$. In this case,  $v_6v_5v_1v_2v_3v_8v_6$ is a $6$-cycle, a contradiction.
 So far, we have proved that $v_8, v_9\notin\{v_1, \ldots, v_7\}$.

 Assume that $v_3$ is incident with $f, f_3$ and $f_4$ in clockwise order. Since $G$ has no $6$-cycle, by Lemma \ref{lem0}, $f_4$ is a $5$-face. Let $f_4=[v_4v_3v_8v_{10}v_{11}]$. If $v_{10}, v_{11}\notin \{v_1, \ldots, v_9\}$, then $G$ contains Configuration (1) or (2) in Fig.1, contrary to Lemma \ref{lem7}. Thus,  assume that $v_{10}\in\{v_1,\ldots, v_9\}$ or $v_{11}\in\{v_1, \ldots, v_9\}$. Since $v_1$ is a $3$-vertex, $v_{10}\neq v_1$ and $v_{11}\neq v_1$. Obviously, $v_{10}\notin\{v_3, v_8, v_9\}$ and $v_{11}\notin\{v_3, v_4, v_5\}$ since $G$ is simple. Since $G$ has no $3$-cycle, $v_{10}\notin \{v_2, v_4\}$ and $v_{11}\notin\{v_2, v_6, v_8\}$.
  If $v_{10}=v_5$, then $v_8v_5v_4v_3v_2v_9v_8$ is a $6$-cycle, a contradiction. If $v_{10}=v_6$, then $v_6v_7v_1v_2v_9v_8v_6$ is a $6$-cycle, a contradiction.
 If $v_{10}=v_7$, then $v_8v_7v_1v_5v_4v_3v_8$ is a $6$-cycle, a contradiction. Thus $v_{10}\notin\{v_1, \ldots, v_9\}$.
  If $v_{11}=v_7$, then $v_7v_1v_2v_9v_8v_{10}v_7$ is a $6$-cycle, a contradiction. If $v_{11}=v_9$, then $v_8v_9v_{10}v_8$ is a $3$-cycle, a contradiction.

In the latter case, let $f_3=[v_3v_4v_8v_9]$. We claim that $v_8, v_9\notin\{v_1, \ldots, v_7\}$. Since $v_1$ is a $3$-vertex, $v_9\neq v_1$ and $v_8\neq v_1$.
Since $G$ is  simple, $v_9\notin\{v_2, v_3, v_4\}$ and $v_8\notin\{v_3, v_4, v_5\}$ by Lemma~\ref{lem5}(1). Since $G$ has no $3$-cycle, $v_9\neq v_5$ and $v_8\notin\{ v_2, v_6\}$.
 If $v_9=v_6$, then $v_6v_7v_1v_5v_4v_3v_6$ is a $6$-cycle, a contradiction. If $v_9=v_7$, then $v_5v_6v_7v_3v_2v_1v_5$ is a $6$-cycle, a contradiction. Thus $v_9\notin\{v_1, \ldots, v_7\}$.
If $v_8=v_7$, then $v_7v_6v_5v_4v_3v_9v_7$ is a $6$-cycle, a contradiction. 

Since $G$ has no $6$-cycle, by Lemma \ref{lem0}, $f_4$ is a $5$-face. Let $f_4=[v_2v_3v_9v_{10}v_{11}]$.  If $v_{10}, v_{11}\notin \{v_1, \ldots, v_9\}$, then $G$ contains Configuration (3) or (4) of Fig.1, contrary to  Lemma \ref{lem7}.
Thus, assume that either  $v_{10}\in\{v_1, \ldots,v_9\}$ or $v_{11}\in\{v_1, \ldots, v_9\}$. Since $v_1$ is a $3$-vertex, $v_{10}\neq v_1$. Since $G$ is simple and by Lemma~\ref{lem5}(1),
$v_{10}\notin\{v_3, v_8, v_9\}$ and $v_{11}\notin\{v_1, v_2, v_3\}$ . Since $G$ has no $3$-cycles, $v_{10}\notin \{v_2, v_4\}$ and $v_{11}\notin\{v_4,v_5, v_7, v_9\}$.
 If $v_{10}=v_6$, then $v_6v_7v_1v_2v_3v_9v_6$ is a $6$-cycle, a contradiction. If $v_{10}=v_7$, then $v_9v_8v_4v_5v_6v_7v_9$ is a $6$-cycle, a contradiction. Thus $v_{10}\notin\{v_1, \ldots, v_4, v_6,\ldots, v_9\}$ and  $v_{11}\notin\{v_1, \ldots, v_5, v_7, v_9\}$.
Assume that $v_{11}=v_6$.  If $v_{10}=v_5$, then $G$ contains Configuration (6) of Fig.1, contrary to Lemma \ref{lem7}. Thus, $v_{10}\not=v_5$. So,
 $v_{10}\notin\{v_1, \ldots, v_9\}$. In this case,  $v_{10}v_6v_5v_4v_8v_9v_{10}$ is a $6$-cycle, a contradiction.
Thus, assume that $v_{11}=v_8$. If $v_{10}\notin\{v_1, \ldots, v_9\}$, then $v_8v_9v_{10}v_8$ is a $3$-cycle, a contradiction. If $v_{10}=v_5$, then $v_5v_4v_8v_5$ is a $3$-cycle, a contradiction. Thus $v_{11}\notin\{v_1, \ldots, v_9\}$. If $v_{10}=v_5$, then $v_5v_6v_7v_1v_2v_{11}v_5$ is a $6$-cycle, a contradiction.
 $\Box$

This implies that $G$ is not existence. We have proved Property (II). This completes the proof of  Theorem \ref{th1}(2).

\subsection{Proof of Theorem~\ref{th1}(3)}

In this section, we prove  Theorem \ref{th1}(3).  A $3$-vertex $v$ is {\em bad} if $v$ is incident with one $3$-face and {\em good} otherwise.

Now we define the  discharge rules as follows.
\medskip

\begin{enumerate}[(R1)]
\item Every $5^+$-face sends $\frac{1}{3}$ to each incident good $3$-vertex, $\frac{1}{2}$ to each incident bad $3$-vertex and $\frac{1}{3}$ to each incident 3-face.
\end{enumerate}

\medskip

It suffices to show that the new weight function $\mu'$ satisfies Properties (I) and (II). 
We first check $\mu'(v)\geq0$ for all $v\in V(G)$. By Lemma \ref{lem5} (1), $d(v)\geq3$.
\begin{enumerate}
\item  $d(v)=3$. If $v$ is  bad, then $v$ is incident with two $7^+$-faces. By (R1),
$\mu'(v)\geq-1+2\times\frac{1}{2}=0$. If $v$ is good, then $v$ is incident with three $5^+$-faces. By (R1),
$\mu'(v)\geq-1+3\times\frac{1}{3}=0$.

\item $d(v)=4$. Since no 4-vertex is involved in discharge procedure,  $\mu'(v)=\mu(v)=4-4=0$.

\item $d(v)\geq5$. By (R1), $\mu'(v)=\mu(v)=d(v)-4\geq1>0$.
\end{enumerate}
Further we check $\mu'(f)\geq0$ for all $f\in F(G)$.
\begin{enumerate}
\item  $d(f)=3$. Then $\mu(f)=-1$. Since $G$ has no 4-cycle, each face adjacent to $f$ is a $5^+$-face. By (R1), $\mu'(v)\geq-1+3\times\frac{1}{3}=0$.

\item  $d(f)=5$. Then $\mu(f)=1$. Since $G$ has no $6$-cycle,  $f$ is not adjacent to $3$-face. By Lemma~\ref{lem5}(2), $f$ is adjacent to at most two  3-vertices. Since $G$ has no $6$-cycles, no $3$-cycle is not adjacent to any $5$-cycle. Thus, $f$ is adjacent to at most two good 3-vertices.  By (R1), $\mu'(f)\geq 1-2\times\frac{1}{3}=\frac{1}{3}>0$.

\item  $d(f)\geq7$.  Let $f$ be incident with $m$ 3-vertices. Since $G$ has no $4$-cycles, no $3$-face is  adjacent to any $3$-face. Then  $v$ is incident with at most $d(f)-m$ 3-faces.
By Lemma \ref{lem5} (2), $m\leq\lfloor\frac{d(f)}{2}\rfloor$. Thus $\mu'(f)\geq d(f)-4-\frac{m}{2}-\frac{1}{3}(d(f)-m)=\frac{2}{3}d(f)-\frac{1}{6}m-4\geq\frac{7}{12}d(f)-4\geq\frac{1}{12}>0$ by (R1).
\end{enumerate}
So far, we have proved Property (I). Assume that Property (II) does not hold. This implies that $\mu'(x)=0$ for all $x\in V(G)\cup F(G)$. Observing above proof, we obtain the following statements.

(a) For each vertex $v\in V(G)$, $3\le d(v)\le4$;

(b) For each face $f\in F (G)$, $d(f)=3$.\\
By (b), $G$ is  one $3$-cycle $[uvw]$. Clearly, $G$ has a matching $M=\{uv\}$ such that $G-M$ is $(2, 1)$ decomposable,  a contradiction. This completes the proof of  Theorem \ref{th1}(3).

\end{document}